\documentclass{article}
\usepackage[affil-it]{authblk}
\usepackage[utf8]{inputenc}
\usepackage{amsmath, amsfonts, amsthm, amssymb, latexsym}
\usepackage{geometry}
\usepackage{enumerate}
\usepackage{multirow}
\usepackage{rotating}
\usepackage{graphicx,color}
\usepackage{float}
\usepackage{mathrsfs}

\newtheorem{thm}{Theorem}[section]
\newtheorem{lemma}[thm]{Lemma}

\newtheorem{define}[thm]{Definition}
\newtheorem{prop}[thm]{Proposition}
\newtheorem{cor}[thm]{Corollary}

\newtheorem*{remark}{Remark}

\title{A fractional parabolic inverse problem involving a time-dependent magnetic potential}
\author{Li Li}
\affil{Department of Mathematics, University of Washington,\\
Seattle, WA 98195, USA}
\date{}

\begin{document}

\maketitle

\noindent \textbf{ABSTRACT.}\, 
We study a class of fractional parabolic equations involving a time-dependent magnetic potential and formulate the corresponding inverse problem.
We determine both the magnetic potential and the electric potential from the exterior partial measurements of the Dirichlet-to-Neumann map.

\section{Introduction}
The study of fractional operators has been an active research field in past decades. Differential equations involving fractional derivatives have been
introduced to describe anomalous diffusion and random processes with jumps in physics and probability theory. See for instance, \cite{meerschaert1999multidimensional, meerschaert2002governing, saichev1997fractional}. 

Correspondingly, various kinds of inverse problems associated with fractional operators have been extensively studied so far. 
The study of the inverse problem for space-fractional operators (very different from the one for time-fractional operators studied in \cite{cheng2009uniqueness, kian2018global}),
namely the fractional Calder\'on problem, was initiated in \cite{ghosh2020calderon} where the authors considered 
the exterior Dirichlet problem
$$((-\Delta)^s+ q)u= 0\,\,\, \text{in}\,\,\Omega,\qquad
u|_{\Omega_e}= f$$
and they showed that the electric potential $q$ in $\Omega$ 
can be determined from 
the exterior partial measurements of the Dirichlet-to-Neumann map 
$$\Lambda_q: f\to (-\Delta)^su_f|_{\Omega_e}.$$
See \cite{bhattacharyya2018inverse, cekic2018calder, ghosh2020uniqueness, ghosh2017calderon, ruland2020fractional}
for further studies based on \cite{ghosh2020calderon}.

As variants of the inverse problem introduced in \cite{ghosh2020calderon},
parabolic analogues of the fractional Calder\'on problem have been studied in recent years (see \cite{canuto2001determining, choulli2004conditional} for results for the local parabolic inverse problem). These studies are motivated by problems involving
continuous time random walk (CTRW) where particle jumps and waiting times 
are associated with (time or space) fractional derivatives in the governing equation.
One related work can be found in \cite{lai2020calderon} where the authors studied the inverse problem for the fractional operator 
$$(\partial_t- \Delta)^s+ q.$$ 
Another related work can be found in \cite{ruland2017quantitative} where the authors established the well-posedness
of the initial exterior problem associated with the fractional operator
$$\partial_t+ (- \Delta)^s$$ 
and its Runge approximation property.

In this paper, we study an inverse problem for a fractional operator 
generalizing $\partial_t+ (- \Delta)^s$. Our operator contains a time-dependent space-fractional derivative and our inverse problem 
can be viewed as a parabolic analogue of the the fractional magnetic Calder\'on problem introduced in \cite{li2020calderon,li2020determining}. 
See \cite{ferreira2007determining, krupchyk2014uniqueness, nakamura1995global, sun1993inverse} for results for the local
magnetic Calder\'on problem. Also see \cite{covi2020inverse} for the study of a different fractional magnetic Calder\'on problem.

More precisely, we consider the time-dependent operator 
$\mathcal{R}^s_{A(t)}$, which is formally defined by
\begin{equation}\label{RsAt}
\langle \mathcal{R}^s_{A(t)}u, v\rangle:= 
2\iint(u(x)- R_{A(t)}(x, y)u(y))v(x)K(x,y)\,dxdy
\end{equation}
for each $t$. Here $K$ is a function associated with a heat kernel (see Subsection 2.2 in \cite{li2020calderon} or Section 2 in \cite{ghosh2017calderon}) satisfying
$$K(x, y)= K(y, x),\qquad  {c}/{|x-y|^{n+2s}}\leq K(x, y)\leq {C}/{|x-y|^{n+2s}},$$ 
$A(\cdot, t)$ is a time-dependent real vector-valued magnetic potential and
\begin{equation}\label{cost}
R_{A(t)}(x, y):= \cos((x-y)\cdot A(\frac{x+y}{2}, t)).
\end{equation}
Clearly, the operator $\mathcal{R}^s_{A(t)}$ coincides with the fractional Laplacian $(-\Delta)^s$ when
$A\equiv 0$ and $K(x, y)= c_{n, s}/|x-y|^{n+2s}$.

Under appropriate assumptions on $A$ and the time-dependent electric potential $q(\cdot, t)$, the initial exterior problem
\begin{equation}\label{iepara}
\left\{
\begin{aligned}
\partial_t u+ \mathcal{R}^s_{A(t)} u+ q(t)u &= 0\quad \text{in}\,\,\Omega\times (-T, T)\\
u&= g \quad \text{in}\,\,\Omega_e\times (-T, T)\\
u&= 0\quad \text{in}\,\,\mathbb{R}^n\times \{-T\}.\\
\end{aligned}
\right.
\end{equation}
is well-posed so we can define the solution operator $P_{A, q}: g\to u_g$
and the Dirichlet-to-Neumann map $\Lambda_{A, q}$,  
which is formally given by
\begin{equation}\label{tDN}
\Lambda_{A, q}g:= \mathcal{R}^s_A (u_g)|_{\Omega_e\times (-T, T)}.
\end{equation}

Our goal here is to determine both $A$ and $q$ from the exterior partial measurements of $\Lambda_{A, q}$.

The following theorem is the main result in this paper. It is remarkable that the magnetic potential can only be determined up to a gauge equivalence in the classical magnetic Calder\'on problem while it can be totally determined (up to the sign) in this fractional inverse problem.
\begin{thm}
Suppose $\Omega\subset B_r(0)$ for some constant $r> 0$, $\mathrm{supp}\,A_j(t)\subset\Omega$ for $t\in [-T, T]$, $A_j\in C^2([-T, T]; L^\infty(\mathbb{R}^n))$,
$q_j\in C^2([-T, T]; L^\infty(\Omega))$ and
$q_j\geq c$ for some constant $c> 0$, 
$W_j$ are open sets s.t. $W_j\cap B_{3r}(0)= \emptyset$
($j= 1, 2$). Let
$$W^{(1, 2)}= \{\frac{x+ y}{2}: x\in W_1, y\in W_2\}.$$
Also assume $W^{(1, 2)}\setminus \Omega\neq \emptyset$. If 
\begin{equation}\label{IDN}
\Lambda_{A_1, q_1}g|_{W_2\times (-T, T)}
= \Lambda_{A_2, q_2}g|_{W_2\times (-T, T)}
\end{equation}
for any $g\in C^\infty_c(W_1\times (-T, T))$, then $A_1(t)= \pm A_2(t)$ and $q_1= q_2$ in $\Omega\times (-T, T)$.
\end{thm}
\begin{remark}
The assumptions on $W_j$ seem unnatural but they are necessary. More precisely, we need the assumption $W_j\cap B_{3r}(0)= \emptyset$ to show Runge approximation
properties (see Proposition 2.4 and Proposition 4.2 later) and we need the assumption 
$W^{(1, 2)}\setminus \Omega\neq \emptyset$ to obtain the integral identity in Subsection 4.3.
These assumptions are analogues of the ones in Theorem 1.1 in \cite{li2020determining}.

\end{remark}

The rest of this paper is organized in the following way. In Section 2, we summarize the background knowledge. We show the well-posedness of the initial exterior problem (\ref{iepara}) in Section 3. We introduce the associated Dirichlet-to-Neumann map, prove the Runge approximation property of our fractional operator and the main  theorem in Section 4.~\\

\noindent \textbf{Acknowledgement.} The author is partly supported by National Science Foundation. The author would like to thank Professor Gunther Uhlmann for helpful discussions.

\section{Preliminaries}

Throughout this paper
\begin{itemize}
\item Fix the space dimension $n\geq 2$ and 
the fractional power $0< s< 1$

\item Fix the constant $T> 0$ and $t$ denotes the time variable

\item $\Omega$ denotes a bounded Lipschitz domain and
$\Omega_e:= \mathbb{R}^n\setminus\bar{\Omega}$

\item $B_r(0)$ denotes the open ball centered at the origin with radius $r> 0$
in $\mathbb{R}^n$

\item If $u(\cdot, t)$ is an $(n+1)$-variable function, then
$u(t)$ denotes the corresponding $n$-variable function for each $t$

\item $A(\cdot, t)$ denotes a time-dependent $\mathbb{R}^n$-valued magnetic potential and $q(\cdot, t)$ denotes a time-dependent electric potential

\item If $A(t)\in L^\infty(\Omega)$, then identify $A(t)$
with its zero extension in $L^\infty(\mathbb{R}^n)$

\item $c, C, C', C_1,\cdots$ denote positive constants (which may depend on some parameters)

\item $\int\cdots\int= \int_{\mathbb{R}^n}\cdots\int_{\mathbb{R}^n}$

\item $X^*$ denotes the continuous dual space of $X$ and write
$\langle f, u\rangle= f(u)$ for $u\in X,\,f\in X^*$
when $X$ is an $n$-variable function space.
\end{itemize}

\subsection{Function spaces}

Throughout this paper we refer all function spaces to real-valued function spaces.

For $\alpha\in \mathbb{R}$, 
$H^\alpha(\mathbb{R}^n)$ denotes the Sobolev space 
$W^{\alpha, 2}(\mathbb{R}^n)$.

We have the natural identification
$$H^{-\alpha}(\mathbb{R}^n)= H^\alpha(\mathbb{R}^n)^*.$$
Let $U$ be an open set in $\mathbb{R}^n$. Let $F$ be a closed set in $\mathbb{R}^n$. Then
$$H^\alpha(U):= \{u|_U: u\in H^\alpha(\mathbb{R}^n)\},\qquad 
H^\alpha_F(\mathbb{R}^n):= 
\{u\in H^\alpha(\mathbb{R}^n): \mathrm{supp}\,u\subset F\},$$
$$\tilde{H}^\alpha(U):= 
\mathrm{the\,\,closure\,\,of}\,\, C^\infty_c(U)\,\,\mathrm{in}\,\, H^\alpha(\mathbb{R}^n).$$
$\Omega$ is Lipschitz bounded implies
$$\tilde{H}^\alpha(\Omega)= H^\alpha_{\bar{\Omega}}(\mathbb{R}^n).$$

Let $X$ be a Banach space. For $m\in \mathbb{N}$, we use $C^m([-T, T]; X)$ (resp. $AC([-T, T]; X)$)
to denote the space consisting of the corresponding Banach space-valued  continuously differentiable (resp. absolutely continuous) functions on $[-T, T]$.

$L^p(-T, T; X)$ denotes the space consisting of the corresponding Banach space-valued $L^p$ functions, equipped with the standard norm
$$||u||_{L^p(-T, T; X)}:= (\int^T_{-T}||u(t)||^p_X\,dt)^{1/p}.$$

\subsection{The operator $\mathcal{R}^s_{A(t)}$}
Recall that we gave the formal definition of $\mathcal{R}^s_{A(t)}$
in (\ref{RsAt}). Note that
$$\iint(u(x)- R_{A(t)}(x, y)u(y))v(x)K(x,y)\,dxdy
= \iint(u(y)- R_{A(t)}(x, y)u(x))v(y)K(x,y)\,dydx$$
so we have
$$\langle \mathcal{R}^s_{A(t)}u, v\rangle=
\iint[(u(x)- R_{A(t)}(x, y)u(y))v(x)+ (u(y)- R_{A(t)}(x, y)u(x))v(y)]K(x,y)\,dxdy
$$
$$= \mathrm{Re}\iint
(u(x)-e^{i(x-y)\cdot A(\frac{x+y}{2}, t)}u(y))
(v(x)-e^{-i(x-y)\cdot A(\frac{x+y}{2}, t)}v(y))K(x, y)\,dxdy$$
and 
\begin{equation}
\langle \mathcal{R}^s_{A(t)}u, v\rangle
= \langle \mathcal{R}^s_{A(t)}v, u\rangle.
\end{equation}

The following lemma is a time-dependent version of Lemma 3.3 in \cite{li2020calderon}.
\begin{lemma}
Suppose $A\in C([-T, T], L^\infty(\mathbb{R}^n))$, then for 
$u\in H^s(\mathbb{R}^n)$, we have
$$c||u||_{H^s}\leq ||u||_{H^s_{A(t)}}\leq C||u||_{H^s}$$
where the magnetic Sobolev norm $||\cdot||_{H^s_{A(t)}}$ is defined by
$$||u||_{H^s_{A(t)}}:= (||u||^2_{L^2}+ 
\langle \mathcal{R}^s_{A(t)}u, u\rangle)^{1/2}$$
and $c, C$ depend on $\sup_{t\in [-T, T]}||A(t)||_{L^\infty} $ but do not depend on $t$.
\end{lemma}

\begin{define}
We define the time-dependent bilinear form associated with $A, q$ by
\begin{equation}\label{tdbf}
B_t[u, v]:= \langle \mathcal{R}^s_{A(t)}u, v\rangle+ \int_\Omega q(t)uv,\qquad
t\in [-T, T].
\end{equation}
\end{define}

The symmetry of $B_t$ follows immediately from the symmetry of $\mathcal{R}^s_{A(t)}$.

The following estimates will be useful when we show the well-posedness of the initial exterior problem later.
\begin{lemma}
Suppose $A, q\in C^2([-T, T]; L^\infty(\Omega))$ and $q\geq c'$ in  $\Omega\times [-T, T]$ for some constant $c'> 0$. Then
\begin{equation}\label{Hsbd}
|B_t[u, v]|\leq C_0||u||_{H^s}||v||_{H^s},\qquad u, v\in H^s(\mathbb{R}^n)
\end{equation}
and for $u, v\in \tilde{H}^s(\Omega)$, we have
\begin{equation}\label{bdcoer}
|B_t[u, u]|\geq c_0||u||^2_{H^s},
\end{equation}
\begin{equation}\label{1stest}
|B_{t+h}[u, v]- B_t[u, v]|\leq C_1h||u||_{H^s}||v||_{H^s},
\end{equation}
\begin{equation}\label{2ndest}
|B_{t+h}[u, v]+ B_{t-h}[u, v]- 2B_t[u, v]|\leq C_2h^2||u||_{H^s}||v||_{H^s}
\end{equation}
where the constants $c_0, C_0, C_1, C_2$ do not depend on $u, v, t$ and $h> 0$.
\end{lemma}

\begin{proof}
(\ref{Hsbd}) and (\ref{bdcoer}) follow from Lemma 2.1 immediately. 

Note that for $u, v\in \tilde{H}^s(\Omega)$, we have
$$\frac{d}{dt}\langle \mathcal{R}^s_{A(t)}u, v\rangle= 
-2\iint \frac{d}{dt}R_{A(t)}(x, y)u(y)v(x)K(x,y)\,dxdy$$
$$= -2\int_\Omega\int_\Omega\frac{d}{dt}R_{A(t)}(x, y)u(y)v(x)K(x,y)\,dxdy$$
$$= 2\int_\Omega\int_\Omega \sin((x-y)\cdot A(\frac{x+y}{2}, t))
(x-y)\cdot \partial_t A(\frac{x+y}{2}, t) u(y)v(x)K(x,y)\,dxdy$$
so we have
$$|\frac{d}{dt}\langle \mathcal{R}^s_{A(t)}u, v\rangle|\leq
C\int_\Omega\int_\Omega\frac{|u(y)v(x)|}{|x-y|^{n+2s-2}}\,dxdy$$
$$\leq C(\int_\Omega\int_\Omega
\frac{|u(y)|^2}{|x-y|^{n+2s-2}}\,dxdy)^{\frac{1}{2}}
(\int_\Omega\int_\Omega\frac{|v(x)|^2}{|x-y|^{n+2s-2}}\,dxdy)^{\frac{1}{2}}.$$
Note that
$$\int_\Omega\int_\Omega
\frac{|u(y)|^2}{|x-y|^{n+2s-2}}\,dxdy=
\int_\Omega(\int_\Omega\frac{1}{|x-y|^{n+2s-2}}\,dx)|u(y)|^2\,dy
\leq C'||u||^2_{L^2}$$
and similarly 
$$\int_\Omega\int_\Omega
\frac{|v(x)|^2}{|x-y|^{n+2s-2}}\,dxdy
\leq C'||v||^2_{L^2}$$
so we have
$$|\frac{d}{dt}\langle \mathcal{R}^s_{A(t)}u, v\rangle|\leq
C''||u||_{L^2}||v||_{L^2}\leq C''||u||_{H^s}||v||_{H^s},$$
which implies (\ref{1stest}) holds.

Also note that 
$$\frac{d^2}{dt^2}R_{A(t)}(x, y)=
-\sin((x-y)\cdot A(\frac{x+y}{2}, t))
(x-y)\cdot \partial_{tt} A(\frac{x+y}{2}, t)$$
$$- \cos((x-y)\cdot A(\frac{x+y}{2}, t))
((x-y)\cdot \partial_t A(\frac{x+y}{2}, t))^2$$
so we have
$$|\frac{d^2}{dt^2}R_{A(t)}(x, y)|\leq C'''|x- y|^2$$
and we can similarly show that 
$$|\frac{d^2}{dt^2}\langle \mathcal{R}^s_{A(t)}u, v\rangle|\leq C''''||u||_{H^s}||v||_{H^s},\qquad u, v\in \tilde{H}^s(\Omega),$$
which implies (\ref{2ndest}) holds.
\end{proof}

We will use the following proposition to prove 
the Runge approximation property later.
\begin{prop}
(Proposition 2.4 in \cite{li2020semilinear}) Suppose $\Omega\cup\mathrm{supp}\,A(t)\subset B_r(0)$ for some $r> 0$, $W$ is an open set s.t. $W\cap B_{3r}(0)\neq \emptyset$. If
$$u\in \tilde{H}^s(\Omega),\qquad \mathcal{R}^s_{A(t)}u|_W= 0$$
then $u= 0$ in $\mathbb{R}^n$.
\end{prop}

\section{Initial Exterior Problem}
From now on we always assume $A, q\in C^2([-T, T]; L^\infty(\Omega))$ and $q\geq c'$ in  $\Omega\times [-T, T]$ for some $c'$.

\subsection{Discretization in time}
First we study the initial value problem
\begin{equation}\label{para}
\left\{
\begin{aligned}
\partial_t u+ \mathcal{R}^s_{A(t)} u+ q(t)u &= f\quad \text{in}\,\,\Omega\times (-T, T)\\
u&= 0\quad \text{in}\,\,\Omega\times \{-T\}.\\
\end{aligned}
\right.
\end{equation}

\begin{prop}
Suppose $f\in L^2(\Omega\times (-T, T))$ and 
\begin{equation}\label{1stf}
||f(t+h)- f(t)||_{L^2(\Omega)}\leq Ch
\end{equation}
for some $C$ independent of $t, h$,
then (\ref{para}) has a unique (weak) solution satisfying
$$u\in L^2(-T, T; \tilde{H}^s(\Omega))\cap AC([-T, T]; L^2(\Omega)),\quad
\partial_t u\in L^2(\Omega\times (-T, T)).$$
\end{prop}

\begin{remark}
The initial value problem associated with $\partial_t+ (-\Delta)^s$ has been studied in \cite{ruland2017quantitative} where the authors used a Galerkin approximation to show the existence of solutions. Here the time-dependent fractional operator 
$\mathcal{R}^s_{A(t)}$ makes the problem much more complicated. We will use the method of discretization in time instead to show the existence of solutions, which can be viewed as a nonlocal analogue of the Rothe's method for local parabolic problems (see Chapter 15 in \cite{rektorys1982method}).
\end{remark}

The following proof relies on the two lemmas in Appendix.
\begin{proof}
\textbf{Existence}: Divide $[-T, T]$ into $p$ subintervals of length $h= 
\frac{2T}{p}< \frac{1}{2}$ and let $t_j= -T+ jh$. 

Consider the discretization in $t$
\begin{equation}\label{discretpara}
\left\{
\begin{aligned}
\frac{z_j- z_{j-1}}{h}+ \mathcal{R}^s_{A(t_j)}z_j+ q(t_j)z_j &= f(t_j)\quad 
j= 1,\cdots, p\\
z_0&= 0.\\
\end{aligned}
\right.
\end{equation}
We can iteratively determine $z_j\in \tilde{H}^s(\Omega)$, which solves the elliptic equation
$$\mathcal{R}^s_{A(t_j)}z_j+ (q(t_j)+ \frac{1}{h})z_j = f(t_j)+ \frac{z_{j-1}}{h}.$$
((\ref{Hsbd}) and (\ref{bdcoer}) ensure the existence and uniqueness of $z_j$
by Lax-Milgram Theorem.)

Define $Z_j:= (z_j- z_{j-1})/h$ and $u^{(1)}: [-T, T]\to \tilde{H}^s(\Omega)$
given by
$$u^{(1)}(t):= z_{j-1}+ Z_j(t- t_{j-1}),\qquad t\in [t_{j-1}, t_j].$$

Now divide $[-T, T]$ into $2^{m-1}p$ subintervals of length $h_m= 2T/(2^{m-1}p)$ and let $t^{(m)}_j= -T+ jh_m$. Similarly, we consider the discretization,
obtain a sequence 
$$\{z^{(m)}_0= 0,\cdots, z^{(m)}_{2^{m-1}p}\}$$
in $\tilde{H}^s(\Omega)$, define $Z^{(m)}_j:= (z^{(m)}_j- z^{(m)}_{j-1})/h_m$ and  
$$u^{(m)}(t):= z^{(m)}_{j-1}+ Z^{(m)}_j(t- t^{(m)}_{j-1}),\qquad
t\in [t^{(m)}_{j-1}, t^{(m)}_j].$$ 

Also define the step functions
$$\tilde{u}^{(m)}(t):= z^{(m)}_j,\quad
\tilde{U}^{(m)}(t):= Z^{(m)}_j,\quad\quad t\in (t^{(m)}_{j-1}, t^{(m)}_j).$$

Note that the constants in Lemma A.2 do not depend on $h$ so for general $m$,
$$||z^{(m)}_j||_{H^s}\leq c_2,\qquad ||Z^{(m)}_j||_{L^2}\leq c_3$$
hold. This implies the boundedness of $\{u^{(m)}\}$ in 
$L^2(-T, T; \tilde{H}^s(\Omega))$ and the boundedness of $\{\tilde{U}^{(m)}\}$
in $L^2(\Omega\times (-T, T))$. 
Hence we can choose weakly convergent sequences s.t.
$$u^{(m_k)}\rightharpoonup u\quad\text{in}\,\,L^2(-T, T; \tilde{H}^s(\Omega)),
\quad
\tilde{U}^{(m_k)}\rightharpoonup \tilde{U}\quad\text{in}\,\,L^2(\Omega\times (-T, T)).$$
Note that $\tilde{U}^{(m)}$ is the weak derivative as well as the pointwise derivative of $u^{(m)}$ so we let $k\to \infty$ in
$$u^{(m_k)}(t)= \int^t_{-T}\tilde{U}^{(m_k)}(\tau)\,d\tau$$
to obtain
$$u(t)= \int^t_{-T}\tilde{U}(\tau)\,d\tau.$$
Hence $u$ is absolutely continuous in $t$, $u(-T)= 0$ and $\partial_t u(t)= \tilde{U}(t)$.

Now we show that this $u$ satisfies the equation in (\ref{para}).

Define the step function 
$$\tilde{f}^{(m)}(t):= f(t^{(m)}_j),\qquad t\in (t^{(m)}_{j-1}, t^{(m)}_j)$$
and the step bilinear form
$$B^{(m)}_t[\cdot, \cdot]:= B_{t^{(m)}_j}[\cdot, \cdot],
\qquad t\in (t^{(m)}_{j-1}, t^{(m)}_j).$$

Fixing $v\in L^2(-T, T; \tilde{H}^s(\Omega))$, we 
let both sides of the discretized  equation 
$$Z^{(m)}_j+ \mathcal{R}^s_{A(t^{(m)}_j)}z^{(m)}_j+ q(t^{(m)}_j)z^{(m)}_j= f(t^{(m)}_j)$$
act on $v(t)$ for $t\in (t^{(m)}_{j-1}, t^{(m)}_j)$ and integrate from $-T$ to $T$, then we have
\begin{equation}\label{intpara}
\int_{-T}^T\langle\tilde{U}^{(m)}(t), v(t)\rangle\,dt
+ \int_{-T}^T B^{(m)}_t[\tilde{u}^{(m)}(t), v(t)]\,dt
= \int_{-T}^T\langle\tilde{f}^{(m)}(t), v(t)\rangle\,dt.
\end{equation}
(\ref{1stf}) ensures that 
\begin{equation}\label{limf}
\int_{-T}^T\langle\tilde{f}^{(m)}(t), v(t)\rangle\,dt\to 
\int_{-T}^T\langle f(t), v(t)\rangle\,dt.
\end{equation}
The weak convergence of $\tilde{U}^{(m_k)}$ implies
\begin{equation}\label{limUpartial}
\int_{-T}^T\langle\tilde{U}^{(m_k)}(t), v(t)\rangle\,dt\to 
\int_{-T}^T\langle\partial_t u(t), v(t)\rangle\,dt.
\end{equation}
Note that (\ref{Hsbd}) ensures that
$$\int_{-T}^T B_t[\cdot, v(t)]\,dt$$
is a bounded linear functional on $L^2(-T, T; \tilde{H}^s(\Omega))$
and by Lemma A.3 we have the weak convergence of $\{\tilde{u}^{(m_k)}\}
$ so 
\begin{equation}\label{limu}
\int_{-T}^T B_t[\tilde{u}^{(m_k)}(t), v(t)]\,dt\to
\int_{-T}^T B_t[u(t), v(t)]\,dt.
\end{equation}
Also we can show
\begin{equation}\label{limbilinear}
\int_{-T}^T B^{(m_k)}_t[\tilde{u}^{(m_k)}(t), v(t)]\,dt
- \int_{-T}^T B_t[\tilde{u}^{(m_k)}(t), v(t)]\,dt\to 0.
\end{equation}

In fact,  we first assume $v(t)= 1_{[\alpha, \beta]}(t)v_0$ where
$v_0\in \tilde{H}^s(\Omega)$ and $\alpha, \beta$ are endpoints of subintervals 
in the $m_k$-division for some $k$. For each large $m_k$, we write
$\alpha= t^{(m_k)}_{j_1}, \beta= t^{(m_k)}_{j_2}$ for some $j_1, j_2$, then
$$|\int_{-T}^T B^{(m_k)}_t[\tilde{u}^{(m_k)}(t), v(t)]\,dt
- \int_{-T}^T B_t[\tilde{u}^{(m_k)}(t), v(t)]\,dt|$$
$$= |\sum_{j= j_1+1}^{j_2}\int_{t^{(m_k)}_{j-1}}^{t^{(m_k)}_j}
B_{t^{(m_k)}_j}[z^{(m_k)}_j, v_0]- B_t[z^{(m_k)}_j, v_0]\,dt|$$
$$\leq \sum_{j= j_1+1}^{j_2}\int_{t^{(m_k)}_{j-1}}^{t^{(m_k)}_j}
C_1||z^{(m_k)}_j||_{H^s}||v_0||_{H^s}h_{m_k}\,dt\leq 2TC_1c_2h_{m_k}||v_0||_{H^s}$$
by using (\ref{1stest}) and the boundedness of $\{z^{(m)}_j\}$.

Since the set consisting of such $v$ 
spans a space dense in $L^2(-T, T; \tilde{H}^s(\Omega))$, we know
(\ref{limbilinear}) holds for all $v\in L^2(-T, T; \tilde{H}^s(\Omega))$.

Combining (\ref{limf}), (\ref{limUpartial}), (\ref{limu}), (\ref{limbilinear}) with (\ref{intpara}), we conclude that
\begin{equation}\label{uintpara}
\int_{-T}^T\langle\partial_t u(t), v(t)\rangle\,dt
+ \int_{-T}^T B_t[u(t), v(t)]\,dt
= \int_{-T}^T\langle f(t), v(t)\rangle\,dt.
\end{equation}

\textbf{Uniqueness}: Let $v= u$ be the solution constructed in the existence part, then 
(\ref{uintpara}) becomes
$$\frac{1}{2}||u(T)||^2_{L^2(\Omega)}+ \int_{-T}^T B_t[u(t), u(t)]\,dt
= \int_{-T}^T\langle f(t), u(t)\rangle\,dt,$$
which implies
$$||u||_{L^2(-T, T; \tilde{H}^s(\Omega))}\leq 
C'||f||_{L^2(-T, T; H^{-s}(\Omega))}.$$
By (\ref{para}) and (\ref{Hsbd}),
$$||\partial_t u||_{L^2(-T, T; H^{-s}(\Omega))}\leq
||f||_{L^2(-T, T; H^{-s}(\Omega))}+ ||\mathcal{R}^s_A u+ qu||_{L^2(-T, T; H^{-s}(\Omega))}$$
$$\leq ||f||_{L^2(-T, T; H^{-s}(\Omega))}+ 
C''||u||_{L^2(-T, T; \tilde{H}^s(\Omega))}.$$
Hence we have
\begin{equation}\label{utfest}
||u||_{L^2(-T, T; \tilde{H}^s(\Omega))}+ 
||\partial_t u||_{L^2(-T, T; H^{-s}(\Omega))}
\leq ||f||_{L^2(-T, T; H^{-s}(\Omega))}.
\end{equation}
The uniqueness is clear when we let $f= 0$.
\end{proof}

\subsection{Well-posedness}
Now we consider (\ref{para}) for general 
$f\in L^2(\Omega\times (-T, T))$.

In fact, we can choose $f_m$ satisfying
(\ref{1stf}) s.t. $f_m\to f$ in $L^2(\Omega\times (-T, T))$. 

Let $u_m$ be the solution corresponding to $f_m$ in (\ref{para}), then
we have
$$||u_m- u_l||_{L^2(-T, T; \tilde{H}^s(\Omega))}+ 
||\partial_t (u_m- u_l)||_{L^2(-T, T; H^{-s}(\Omega))}
\leq ||f_m- f_l||_{L^2(-T, T; H^{-s}(\Omega))}$$
so 
$$u_m\to u\quad\text{in}\,\,L^2(-T, T; \tilde{H}^s(\Omega)),\quad
\partial_t u_m\to v\quad\text{in}\,\,L^2(-T, T; H^{-s}(\Omega))$$
for some $u, v$ and $\partial_t u= v$. 

This implies the convergence of $\{u_m\}$ in 
$C([-T, T]; L^2(\Omega))$ (see, for instance, Theorem 1 in Section 1.2 in Chapter 18 in \cite{dautary1992mathematical}) and this $u$ satisfies the estimate
(\ref{utfest}). 

Hence we reach the following conclusion.

\begin{cor}
Suppose $f\in L^2(\Omega\times (-T, T))$,
then (\ref{para}) has a unique (weak) solution satisfying
$$u\in L^2(-T, T; \tilde{H}^s(\Omega))\cap C([-T, T]; L^2(\Omega)),\quad 
\partial_t u\in L^2(-T, T; H^{-s}(\Omega)).$$
\end{cor}

From now on we always assume $\Omega\subset B_r(0)$ for some constant $r> 0$ and $W$ is an open set in $\mathbb{R}^n$ s.t. $W\cap B_{3r}(0)= \emptyset$.

\begin{prop}
Suppose $g\in C^\infty_c(W\times (-T, T))$, then (\ref{iepara}) has a unique (weak) solution $u$ satisfying 
$$w\in L^2(-T, T; \tilde{H}^s(\Omega))\cap C([-T, T]; L^2(\Omega)),\quad
\partial_t w\in L^2(-T, T; H^{-s}(\Omega))$$
where $w:= u- g$.
\end{prop}
\begin{proof}
By (\ref{RsAt}), we have
$$\mathcal{R}^s_A g|_{\Omega\times (-T, T)}
= (-\Delta)^s g|_{\Omega\times (-T, T)}$$
where $(-\Delta)^s$ acts on space variables. This is because $(x, y)\in \Omega\times W$ implies $|(x+y)/2|\geq r$ and thus $A(t)= 0$,
$R_{A(t)}= 1$.
Consider the problem
\begin{equation}
\left\{
\begin{aligned}
\partial_t w+ \mathcal{R}^s_{A(t)} w+ q(t)w &= f\quad \text{in}\,\,\Omega\times (-T, T)\\
u&= 0\quad \text{in}\,\,\Omega\times \{-T\}\\
\end{aligned}
\right.
\end{equation}
where $f:= -\mathcal{R}^s_A g|_{\Omega\times (-T, T)}$. Since 
$$(-\Delta)^s: H^\alpha(\mathbb{R}^n)\to H^{\alpha-2s}(\mathbb{R}^n),\qquad
\alpha\in \mathbb{R}$$ 
(see Lemma 2.1 in \cite{ghosh2020calderon}), it is clear that $f\in L^2(\Omega\times (-T, T))$. Now we apply Corollary 3.2 to complete the proof.
\end{proof}

Consider the substitutions
$\tilde{A}(x, t):= A(x, -t)$, $\tilde{q}(x, t):= q(x, -t)$, 
$\tilde{g}(x, t):= g(x, -t)$ and $\tilde{u}(x, t):= u(x, -t)$. Then we know the proposition above holds for the dual problem
\begin{equation}\label{ieparadual}
\left\{
\begin{aligned}
-\partial_t u+ \mathcal{R}^s_{A(t)} u+ q(t)u &= 0\quad \text{in}\,\,\Omega\times (-T, T)\\
u&= g \quad \text{in}\,\,\Omega_e\times (-T, T)\\
u&= 0\quad \text{in}\,\,\mathbb{R}^n\times \{T\}.\\
\end{aligned}
\right.
\end{equation}

\begin{define}
We denote the solution operator $g\to u_g$ associated with (\ref{iepara}) by $P_{A, q}$ and we denote the solution operator associated with the dual problem (\ref{ieparadual}) by $P^*_{A, q}$.
\end{define}

\section{Inverse Problem}
\subsection{Dirichlet-to-Neumann map}
Proposition 3.3 ensures that the Dirichlet-to-Neumann map $\Lambda_{A, q}$
given by (\ref{tDN}) is well-defined at least for 
$g\in C^\infty_c(W\times (-T, T))$.

Now let $g\in C^\infty_c(W_1\times (-T, T))$ and
$h\in C^\infty_c(W_2\times (-T, T))$.

By the definition of the solution operator $P_{A, q}$ we have
\begin{equation}\label{DNint1}
\int^T_{-T}\langle \Lambda_{A, q}g(t), h(t)\rangle\,dt
= \int^T_{-T}B_t[u(t), \tilde{h}(t)]\,dt
+ \int^T_{-T}\langle\partial_t u(t), \tilde{h}(t)\rangle_\Omega\,dt\
\end{equation}
for any $\tilde{h}$ satisfying 
$\tilde{h}- h\in L^2(-T, T; \tilde{H}^s(\Omega))$.
Here $u:= P_{A, q}g$, $w:= u- g$ and 
$$\langle\partial_t u(t), \tilde{h}(t)\rangle_\Omega:=
\langle\partial_t w(t), \tilde{h}(t)- h(t)\rangle.$$

Similarly we can define
$$\Lambda^*_{A, q}h:= \mathcal{R}^s_A u^*|_{\Omega_e\times (-T, T)}$$
where $u^*:= P^*_{A, q}h$ and we have
\begin{equation}\label{DNint2}
\int^T_{-T}\langle \Lambda^*_{A, q}h(t), g(t)\rangle\,dt
= \int^T_{-T}B_t[u^*(t), \tilde{g}(t)]\,dt
+ \int^T_{-T}\langle-\partial_t u^*(t), \tilde{g}(t)\rangle_\Omega\,dt
\end{equation}
for any $\tilde{g}$ satisfying 
$\tilde{g}- g\in L^2(-T, T; \tilde{H}^s(\Omega))$.

\begin{prop}
For $g\in C^\infty_c(W_1\times (-T, T))$ and
$h\in C^\infty_c(W_2\times (-T, T))$, we have
\begin{equation}\label{dualsym}
\int^T_{-T}\langle \Lambda_{A, q}g(t), h(t)\rangle\,dt
= \int^T_{-T}\langle \Lambda^*_{A, q}h(t), g(t)\rangle\,dt.
\end{equation}
\end{prop}
\begin{proof}
Let $\tilde{h}= u^*$ in (\ref{DNint1}) and 
let $\tilde{g}= u$ in (\ref{DNint2}).
Since $u(-T)= u^*(T)= 0$, we have
$$\int^T_{-T}\langle \Lambda_{A, q}g(t), h(t)\rangle\,dt
- \int^T_{-T}\langle \Lambda^*_{A, q}h(t), g(t)\rangle\,dt$$
$$= \int^T_{-T}\langle\partial_t u(t), u^*(t)\rangle_\Omega
+ \langle\partial_t u^*(t), u(t)\rangle_\Omega\,dt 
=\langle u(t), u^*(t)\rangle_\Omega|^{t= T}_{t= -T}= 0.$$
\end{proof}

Now we build the integral identity, which will be used in the proof of the main theorem.

For $g_j\in C^\infty_c(W_j\times (-T, T))$ ($j= 1, 2$), 
let $u_1= P_{A_1, q_1}(g_1)$
and $u^*_2= P^*_{A_2, q_2}(g_2)$, i.e.
$u_1$ is the unique weak solution of
\begin{equation}
\left\{
\begin{aligned}
\partial_t u+ \mathcal{R}^s_{A_1(t)} u+ q_1(t)u &= 0\quad \text{in}\,\,\Omega\times (-T, T)\\
u&= g_1 \quad \text{in}\,\,\Omega_e\times (-T, T)\\
u&= 0\quad \text{in}\,\,\mathbb{R}^n\times \{-T\}.\\
\end{aligned}
\right.
\end{equation}
and $u^*_2$ is the unique weak solution of
\begin{equation}
\left\{
\begin{aligned}
-\partial_t u+ \mathcal{R}^s_{A_2(t)} u+ q_2(t)u &= 0\quad \text{in}\,\,\Omega\times (-T, T)\\
u&= g_2 \quad \text{in}\,\,\Omega_e\times (-T, T)\\
u&= 0\quad \text{in}\,\,\mathbb{R}^n\times \{T\},\\
\end{aligned}
\right.
\end{equation}
then we have
$$\int^T_{-T}\langle \Lambda_{A_1, q_1}g_1(t), g_2(t)\rangle-
\langle\Lambda_{A_2, q_2}g_1(t), g_2(t)\rangle\,dt$$
$$= \int^T_{-T}\langle \Lambda_{A_1, q_1}g_1(t), g_2(t)\rangle\,dt
-\int^T_{-T}\langle \Lambda^*_{A_2, q_2}g_2(t), g_1(t)\rangle\,dt$$
$$= \int^T_{-T}B^{(1)}_t[u_1(t), u^*_2(t)]
+ \langle\partial_t u_1(t), u^*_2(t)\rangle_\Omega\,dt
- \int^T_{-T}B^{(2)}_t[u^*_2(t), u_1(t)]
+ \langle-\partial_t u^*_2(t), u_1(t)\rangle_\Omega\,dt$$
$$=\int^T_{-T}B^{(1)}_t[u_1(t), u^*_2(t)]\,dt
- \int^T_{-T}B^{(2)}_t[u_1(t), u^*_2(t)]\,dt$$
\begin{equation}\label{intdif}
= \int^T_{-T}\iint G(x, y, t)u_1(y, t)u^*_2(x, t)
-\int^T_{-T}\int_\Omega(q_2- q_1)u_1u^*_2
\end{equation}
where
$$G(x, y, t):= 2(R_{A_2(t)}(x, y)- R_{A_1(t)}(x, y))K(x, y).$$

\subsection{Runge approximation}
\begin{prop}
Suppose $\Omega\subset B_r(0)$ for some constant $r> 0$ and $W$ is an open set in $\mathbb{R}^n$ s.t. $W\cap B_{3r}(0)= \emptyset$, then 
$$S:= \{P_{A, q}g|_{\Omega\times (-T, T)}: 
g\in C^\infty_c(W\times (-T, T))\},$$
$$S^*:= \{P^*_{A, q}g|_{\Omega\times (-T, T)}: 
g\in C^\infty_c(W\times (-T, T))\}$$
are dense in $L^2(\Omega\times (-T, T))$.
\end{prop}
\begin{proof}
By the Hahn-Banach Theorem, it suffices to show that:

If $v\in L^2(\Omega\times (-T, T))$ and $\int^T_{-T}\int_\Omega vw= 0$ for all $w\in S$, then $v= 0$ in $\Omega\times (-T, T)$.

In fact, for a given $v\in L^2(\Omega\times (-T, T))$, 
let $\phi\in L^2(-T, T; \tilde{H}^s(\Omega))$ be the solution of
\begin{equation}
\left\{
\begin{aligned}
-\partial_t \phi+ \mathcal{R}^s_{A(t)} \phi+ q(t)\phi &= v\quad \text{in}\,\,\Omega\times (-T, T)\\
\phi&= 0\quad \text{in}\,\,\Omega\times \{T\}.\\
\end{aligned}
\right.
\end{equation}

For $g\in C^\infty_c(W\times (-T, T))$, 
write $u_g:= P_{A, q}g$, then we have
$$\int^T_{-T}\int_\Omega v u_g= \int^T_{-T}\langle-\partial_t \phi(t)+ 
\mathcal{R}^s_{A(t)} \phi(t)+ q(t)\phi, u_g(t)- g(t)\rangle\,dt$$
$$= \int_{-T}^T\langle\partial_t u_g(t), \phi(t)\rangle
+ B_t[u_g(t), \phi(t)]\,dt
- \int_{-T}^T\langle\mathcal{R}^s_{A(t)}g(t), \phi(t)\rangle\,dt$$
\begin{equation}\label{RAPid}
= -\int_{-T}^T\langle\mathcal{R}^s_{A(t)}\phi(t), g(t)\rangle\,dt.
\end{equation}
The first equality holds since $u_g- g\in L^2(-T, T; \tilde{H}^s(\Omega))$,
the second equality holds since $u_g(-T)= \phi(T)= 0$ and the last equality holds since $\phi\in L^2(-T, T; \tilde{H}^s(\Omega))$ and $u_g$ is the solution of (\ref{iepara}).

Hence, if $\int^T_{-T}\int_\Omega vw= 0$ for all $w\in S$, then 
(\ref{RAPid}) yields
$$\int_{-T}^T\langle\mathcal{R}^s_{A(t)}\phi(t), g(t)\rangle\,dt= 0,\qquad
g\in C^\infty_c(W\times (-T, T))$$
so for each $t$ we have
$$\phi(t)\in \tilde{H}^s(\Omega),
\qquad \mathcal{R}^s_{A(t)}\phi(t)|_W= 0,$$
which implies $\phi(t)= 0$ in $\mathbb{R}^n$ for each $t$ by Proposition 2.4 and thus $v= 0$ in $\Omega\times (-T, T)$. 

Similarly we can show $S^*$ is dense in $L^2(\Omega\times (-T, T))$.
\end{proof}

\begin{remark}
Proposition 4.2 can be viewed as a generalization of Theorem 2 in \cite{ruland2017quantitative}. We refer readers to \cite{dipierro2019local,ruland2017quantitative} for more approximation properties of solutions of nonlocal evolution problems.
\end{remark}

\subsection{Proof of the main theorem}
Now we are ready to prove Theorem 1.1. As in the proof of Theorem 1.1 in \cite{li2020determining},
we exploit the integral identity and the Runge approximation property associated with our operator.

\begin{proof}
Write  $u_1= P_{A_1, q_1}(g_1)$ and $u^*_2= P^*_{A_2, q_2}(g_2)$ 
for $g_j\in C^\infty_c(W_j\times (-T, T))$.

As in the proof of Theorem 1.1 in \cite{li2020determining}, the assumptions on
$W_1, W_2, W^{(1,2)}$ ensure that
$$\iint G(x, y, t)u_1(y, t)u^*_2(x, t)\,dxdy
= \int_\Omega\int_\Omega G(x, y, t)u_1(y, t)u^*_2(x, t)\,dxdy$$
for each $t$ (if we shrink $W_1, W_2$ when necessary).

By the integral identity (\ref{intdif}), (\ref{IDN}) implies 
\begin{equation}\label{Gid}
\int^T_{-T}\int_\Omega\int_\Omega G(x, y, t)u_1(y, t)u^*_2(x, t)
=\int^T_{-T}\int_\Omega(q_2- q_1)u_1u^*_2.
\end{equation}

\textbf{Determine $A$}: We fix open sets $\Omega_j\subset \Omega$ s.t. 
$\Omega_1\cap \Omega_2= \emptyset$. We also fix
$\phi_j\in C^\infty_c(\Omega_j)$ and the constants
$a, b\in (-T, T)$ and $\epsilon > 0$. Write 
$$\tilde{\phi}_j(x, t):= 1_{[a, b]}(t)\phi_j(x).$$

By Proposition 4.2, we can choose 
$g_1\in C^\infty_c(W_1\times (-T, T))$ s.t. 
$$||u_1- \tilde{\phi}_1||_{L^2(\Omega\times (-T, T))}\leq \epsilon$$
and for this chosen $g_1$, 
we can choose $g_2\in C^\infty_c(W_2\times (-T, T))$ s.t.
$$||u_1||_{L^2(\Omega\times (-T, T))}
||u^*_2- \tilde{\phi}_2||_{L^2(\Omega\times (-T, T))}\leq \epsilon.$$

Note that $\phi_1(x)\phi_2(x)= 0$ for $x\in \Omega$ so
$$|\int^T_{-T}\int_\Omega (q_2- q_1)u_1u^*_2|=
|\int^T_{-T}\int_\Omega (q_2- q_1)(u_1-\tilde{\phi}_1)\tilde{\phi}_2+ 
\int^T_{-T}\int_\Omega (q_2- q_1)u_1(u^*_2- \tilde{\phi}_2)|$$
\begin{equation}\label{se}
\leq ||(q_2- q_1)||_{L^\infty}||\tilde{\phi}_2||_{L^2}
||u_1-\tilde{\phi}_1||_{L^2}+
||(q_2- q_1)||_{L^\infty}||u_1||_{L^2}||u^*_2- \tilde{\phi}_2||_{L^2}\leq C\epsilon.
\end{equation}

Also note that
$$|G(x, y, t)|\leq \frac{C'}{|x- y|^{n+ 2s- 2}},$$
which implies
$$\int_\Omega|G(x, y, t)|dy\leq C'',\, x\in \Omega,\qquad
\int_\Omega|G(x, y, t)|dx\leq C'',\, y\in \Omega$$
where $C', C''$ do not depend on $t$.
By the generalized Young's Inequality,
$$||T_t f||_{L^2(\Omega)}\leq C''||f||_{L^2(\Omega)},\qquad (T_t f)(x):= \int_\Omega|G(x, y, t)f(y)|\,dy.$$
Now note that
$$\int^T_{-T}\int_\Omega\int_\Omega G(x, y, t)u_1(y, t)u^*_2(x, t)\,dxdydt
- \int^b_a\int_{\Omega_1}\int_{\Omega_2} G(x, y, t)
\phi_1(y)\phi_2(x)\,dxdydt$$
$$= \int^T_{-T}\int_\Omega\int_\Omega G(x,y,t)(u_1(y, t)-\tilde{\phi}_1(y, t))\tilde{\phi}_2(x, t)\,dxdydt$$
$$+ \int^T_{-T}\int_\Omega\int_\Omega G(x,y, t)u_1(y, t)
(u^*_2(x, t)- \tilde{\phi}_2(x, t))\,dxdydt.$$
By Cauchy-Schwarz inequality, we have the estimate
$$\int^T_{-T}\int_\Omega(\int_\Omega |G(x,y,t)u_1(y, t)|\,dy)
|u^*_2(x, t)- \tilde{\phi}_2(x, t)|\,dxdt$$
$$\leq (\int^T_{-T}\int_\Omega
|(T_tu_1(t))(x)|^2\,dxdt)^{\frac{1}{2}}
||u^*_2- \tilde{\phi}_2||_{L^2(\Omega\times (-T, T))}$$
$$\leq C''(\int^T_{-T}
||u_1(t)||^2_{L^2(\Omega)}\,dt)^{\frac{1}{2}}
||u^*_2- \tilde{\phi}_2||_{L^2(\Omega\times (-T, T))}
= C''||u_1||_{L^2(\Omega\times (-T, T))}
||u^*_2- \tilde{\phi}_2||_{L^2(\Omega\times (-T, T))}.$$
Similarly, we have
$$\int^T_{-T}\int_\Omega(\int_\Omega |G(x,y,t)\tilde{\phi}_2(x, t)|\,dx)
|(u_1(y, t)-\tilde{\phi}_1(y, t))|\,dydt$$
$$\leq C''||\tilde{\phi}_2||_{L^2(\Omega\times (-T, T))}
||u_1- \tilde{\phi}_1||_{L^2(\Omega\times (-T, T))}.$$
Hence
\begin{equation}\label{de}
|\int^T_{-T}\int_\Omega\int_\Omega G(x, y, t)u_1(y, t)u^*_2(x, t)
- \int^b_a\int_{\Omega_1}\int_{\Omega_2} G(x, y, t)\phi_1(y)\phi_2(x)|
\leq C'''\epsilon.
\end{equation}

We combine (\ref{se}), (\ref{de}) with (\ref{Gid}). 
$\epsilon$ is arbitrary implies
$$\int^b_a\int_{\Omega_1}\int_{\Omega_2} G(x, y, t)
\phi_1(y)\phi_2(x)\,dxdydt= 0.$$
Then $[a, b]$ is arbitrary implies
$$\int_{\Omega_1}\int_{\Omega_2} G(x, y, t)
\phi_1(y)\phi_2(x)\,dxdy= 0$$
for each $t$ and thus $G(x, y, t)= 0$ in $\Omega_1\times \Omega_2$ for each $t$ since $\phi_1, \phi_2$ are arbitrary. 
Now we can conclude that $G(x, y, t)= 0$ for $x, y \in\Omega$ whenever $x\neq y$ since $\Omega_1, \Omega_2$ are arbitrary. Hence
\begin{equation}\label{RA12}
R_{A_1(t)}(x, y)= R_{A_2(t)}(x, y),\quad x, y\in\Omega
\end{equation}
for each $t$, which implies $A_1(t)= \pm A_2(t)$ as in the proof of Theorem 1.1 in \cite{li2020determining}.

\textbf{Determine $q$}: Now (\ref{Gid}) becomes
$$\int^T_{-T}\int_\Omega(q_2- q_1)u_1u^*_2= 0.$$
Fixing $\epsilon >0$ and $f\in L^2(\Omega\times (-T, T))$, by the Runge approximation property (Proposition 4.2) we can choose
$g_1\in C^\infty_c(W_1\times (-T, T))$ s.t.
$$||u_1-f||_{L^2(\Omega\times (-T, T))}\leq \epsilon$$
and for this chosen $u_1$, we can choose $g_2\in C^\infty_c(W_2)$ s.t.
$$||u_1||_{L^2(\Omega\times (-T, T))}||u^*_2-1||_{L^2(\Omega\times (-T, T))}
\leq \epsilon.$$
Now we have
$$|\int^T_{-T}\int_\Omega(q_1-q_2)f|= |\int^T_{-T}\int_\Omega(q_1-q_2)(f-u_1)+ \int^T_{-T}\int_\Omega(q_1-q_2)u_1(1-u^*_2)|
\leq C\epsilon.$$
We conclude that $q_1= q_2$ since $\epsilon, f$ are arbitrary.
\end{proof}

\appendix
\section{Appendix}
The following well-known estimates (see, for instance, Remark 15.3 on page 286 in \cite{rektorys1982method}) will be useful when we prove the next lemma. They can be viewed as discrete analogues of Gr\"onwall's inequalities.

\begin{prop}
Let $\alpha_1,\cdots,\alpha_j$ be nonnegative numbers. Let $A, B, h$ be positive constants.\\
(a)\, If $\alpha_1\leq A$ and 
$$\alpha_i\leq A+ Bh(\alpha_1+\cdots +\alpha_{i-1}),\qquad i= 2,\cdots,j$$
then we have
$$\alpha_i\leq Ae^{B(i-1)h},\qquad i= 1,\cdots,j.$$
(b)\, If $\alpha_1\leq A$, $Bh< 1$ and 
$$\alpha_i\leq A+ Bh(\alpha_1+\cdots +\alpha_i),\qquad i= 2,\cdots,j$$
then we have
$$\alpha_i\leq \frac{A}{1- Bh}e^{B(i-1)h/(1-Bh)},\qquad i= 2,\cdots,j.$$
\end{prop}

Now we prove the following two lemmas to complete the proof of Proposition 3.1. 
They are essentially the same as their counterparts in the local parabolic problem (see page 286-294 in \cite{rektorys1982method} for details). We include the proofs here for completeness and convenience of readers.

\begin{lemma}
	$\{z_j\}$ and $\{Z_j\}$ defined in the proof of Proposition 3.1 satisfy
	$$||z_j||_{H^s}\leq c_2,\qquad ||Z_j||_{L^2}\leq c_3$$
	where $c_2, c_3$ do not depend on $h< \frac{1}{2}$.
\end{lemma}
\begin{proof}
	We first show that $\{z_j\}$ is bounded in $L^2(\Omega)$.
	
	In fact, let both sides of the equation in (\ref{discretpara}) act on $z_j$, then we have
	$$B_{t_j}[z_j, z_j]+ \frac{1}{h}\langle z_j- z_{j-1}, z_j\rangle=
	\langle f(t_j), z_j\rangle,$$
	which implies
	$$||z_j||_{L^2}\leq ||z_{j-1}||_{L^2}+ h||f(t_j)||_{L^2},$$
	then iteratively we can show
	$$||z_j||_{L^2}\leq h\sum_{l= 1}^j||f(t_l)||_{L^2}\leq jhC_f\leq 2TC_f=: c_1$$
	where the constant $C_f$ depends on $f$.
	
	Next we show that $\{z_j\}$ is bounded in $\tilde{H}^s(\Omega)$.
	
	In fact, let both sides of the equation in (\ref{discretpara}) act on 
	$z_j- z_{j-1}$, then we have
	\begin{equation}\label{btjj1}
	B_{t_j}[z_j, z_j- z_{j-1}]+ \frac{1}{h}||z_j- z_{j-1}||_{L^2}=
	\langle f(t_j), z_j- z_{j-1}\rangle.
	\end{equation}
	
	Note that
	$$B_{t_l}[z_l, z_l- z_{l-1}]= \frac{1}{2}
	(B_{t_l}[z_l, z_l]+ B_{t_l}[z_l-z_{l-1}, z_l-z_{l-1}]
	- B_{t_l}[z_{l-1}, z_{l-1}]),$$
	then we sum up all the identities in the form (\ref{btjj1}) for $1\leq l\leq j$
	and omit all the non-negative terms
	$$||z_l- z_{l-1}||_{L^2},\qquad B_{t_l}[z_l-z_{l-1}, z_l-z_{l-1}]$$
	to obtain the inequality
	$$B_{t_j}[z_j, z_j]- \sum^{j-1}_{l= 1}
	(B_{t_{l+1}}[z_l, z_l]- B_{t_l}[z_l, z_l]))
	\leq 2(\langle f(t_j), z_j\rangle+ \sum^{j-1}_{l= 1}
	\langle f(t_l)- f(t_{l+1}), z_l\rangle).$$
	By (\ref{bdcoer}), (\ref{1stest}) and (\ref{1stf}), this inequality implies
	$$c_0||z_j||^2_{H^s}\leq C_1h\sum^{j-1}_{l= 1}||z_l||^2_{H^s}
	+ 2(||f(t_j)||_{L^2}||z_j||_{L^2}
	+ \sum^{j-1}_{l= 1}||f(t_l)- f(t_{l+1})||_{L^2}||z_l||_{L^2})$$
	$$\leq C_1h\sum^{j-1}_{l= 1}||z_l||^2_{H^s}
	+ 2c_1C_f+ 2c_1C(j-1)h.$$
	Since $(j-1)h\leq 2T$, we have
	$$||z_j||^2_{H^s}\leq C_2'+ C_1'h\sum^{j-1}_{l= 1}||z_l||^2_{H^s}$$
	where $C_1', C_2'$ do not depend on $h$. By Proposition A.1 (a) this implies
	$$||z_j||^2_{H^s}\leq C_2'e^{C_1'(j-1)h},\qquad 
	||z_j||_{H^s}\leq (C_2'e^{2TC_1'})^{\frac{1}{2}}=: c_2.$$
	
	Now we show $\{Z_j\}$ is bounded in $L^2(\Omega)$.
	
	In fact, we can combine the consecutive equations in (\ref{discretpara}) to get
	$$Z_j- Z_{j-1}+ (\mathcal{R}^s_{A(t_j)}z_j+ q(t_j)z_j)
	- (\mathcal{R}^s_{A(t_{j-1})}z_{j-1}+ q(t_{j-1})z_{j-1})= f(t_j)- f(t_{j-1}).$$
	We let both sides act on $Z_j$, then we have
	\begin{equation}\label{BZ}
	\langle Z_j- Z_{j-1}, Z_j\rangle+ B_{t_j}[z_j, Z_j]- B_{t_{j-1}}[z_{j-1}, Z_j]
	= \langle f(t_j)- f(t_{j-1}), Z_j\rangle.
	\end{equation}
	
	Note that
	$$B_{t_j}[z_j, Z_j]= \frac{1}{2h}(B_{t_j}[z_j, z_j]
	-B_{t_j}[z_{j-1}, z_{j-1}]+ h^2B_{t_j}[Z_j, Z_j]),$$
	$$B_{t_{j-1}}[z_{j-1}, Z_j]= \frac{1}{2h}(B_{t_{j-1}}[z_j, z_j]
	-B_{t_{j-1}}[z_{j-1}, z_{j-1}]- h^2B_{t_{j-1}}[Z_j, Z_j])$$
	so we have
	$$B_{t_j}[z_j, Z_j]- B_{t_{j-1}}[z_{j-1}, Z_j]$$
	$$= \frac{h}{2}(B_{t_j}[Z_j, Z_j]+ B_{t_{j-1}}[Z_j, Z_j])
	+ \frac{1}{2h}(B_{t_j}[z_j, z_j]- B_{t_{j-1}}[z_j, z_j]
	+ B_{t_{j-1}}[z_{j-1}, z_{j-1}]- B_{t_j}[z_{j-1}, z_{j-1}]).$$
	Also note that
	$$\langle Z_j- Z_{j-1}, Z_j\rangle= 
	\frac{1}{2}(||Z_j||^2_{L^2}+ ||Z_j- Z_{j-1}||^2_{L^2}- ||Z_{j-1}||^2_{L^2}),$$
	$$|\langle \frac{f(t_j)- f(t_{j-1})}{h}, Z_j\rangle|
	\leq \frac{1}{2}(||Z_j||^2_{L^2}+ C^2).$$
	
	Now we sum up all the identities in the form (\ref{BZ}) for $2\leq l\leq j$ and 
	omit all the non-negative terms 
	$$B_{t_l}[Z_l, Z_l],\qquad B_{t_{l-1}}[Z_l, Z_l],
	\qquad ||Z_l- Z_{l-1}||^2_{L^2}$$ 
	to obtain the inequality
	$$||Z_j||^2_{L^2}- ||Z_1||^2_{L^2}+ \frac{1}{h}
	(B_{t_j}[z_j, z_j]- B_{t_{j-1}}[z_j, z_j])$$
	$$+\frac{1}{h}\sum^{j-1}_{l= 2}(-B_{t_{l+1}}[z_l, z_l]
	+ 2B_{t_l}[z_l, z_l]- B_{t_{l-1}}[z_l, z_l])
	+\frac{1}{h}(B_{t_1}[z_1, z_1]- B_{t_2}[z_1, z_1])$$
	\begin{equation}\label{Zjest}
	\leq h\sum^j_{l= 2}(||Z_l||^2_{L^2}+ C^2).
	\end{equation}
	
	By (\ref{1stest}) and (\ref{2ndest}), we have
	$$|B_{t_l}[z_l, z_l]- B_{t_{l-1}}[z_l, z_l]|
	\leq C_1h||z_l||^2_{H^s}\leq c^2_2C_1h,$$
	$$|-B_{t_{l+1}}[z_l, z_l]+ 2B_{t_l}[z_l, z_l]- B_{t_{l-1}}[z_l, z_l]|
	\leq C_2h^2||z_l||^2_{H^s}\leq c^2_2C_2h^2$$
	so (\ref{Zjest}) implies
	$$||Z_j||^2_{L^2}\leq ||Z_1||^2_{L^2}+ (j-1)h C^2+ 2c^2_2C_1
	+ (j-2)c^2_2C_2h+ h\sum^j_{l= 2}||Z_l||^2_{L^2}.$$
	Since $||Z_1||_{L^2}\leq C_f$ and $(j-1)h\leq 2T$, we can write
	$$||Z_j||^2_{L^2}\leq C_2''+ h\sum^j_{l= 2}||Z_l||^2_{L^2}$$
	where $C_2''$ does not depend on $h$. Since $h< 1/2$, by Proposition A.1 (b) 
	we have the estimate
	$$||Z_j||^2_{L^2}\leq 2C_2''e^{2(j-1)h},\qquad
	||Z_j||_{L^2}\leq (2C_2''e^{4T})^{\frac{1}{2}}=: c_3.$$
\end{proof}

\begin{lemma}
	$u$ and $\{\tilde{u}^{(m_k)}\}$ defined in the proof of Proposition 3.1 satisfy
	$$\tilde{u}^{(m_k)}\rightharpoonup u\quad\text{in}\,\,L^2(-T, T; \tilde{H}^s(\Omega)).$$
\end{lemma}
\begin{proof}
	Since $u^{(m_k)}\rightharpoonup u$ in $L^2(-T, T; \tilde{H}^s(\Omega)),$
	we only need to show that
	$$u^{(m_k)}- \tilde{u}^{(m_k)}\rightharpoonup 0\quad\text{in}\,\,L^2(-T, T; \tilde{H}^s(\Omega)).$$
	
	Consider $v(t)= 1_{[\alpha, \beta]}(t)v_0$ where
	$v_0\in H^{-s}(\Omega)$ and $\alpha, \beta$ are endpoints of subintervals 
	in the $m_k$-division for some $k$. For each large $m_k$, we write
	$\alpha= t^{(m_k)}_{j_1}, \beta= t^{(m_k)}_{j_2}$ for some $j_1, j_2$, then
	$$\int^T_{-T}\langle v(t),  u^{(m_k)}(t)- \tilde{u}^{(m_k)}(t)\rangle\,dt
	= \sum_{j= j_1+1}^{j_2}\int_{t^{(m_k)}_{j-1}}^{t^{(m_k)}_j}
	\langle v_0,  u^{(m_k)}(t)- \tilde{u}^{(m_k)}(t)\rangle\,dt$$
	$$= \sum_{j= j_1+1}^{j_2}\int_{t^{(m_k)}_{j-1}}^{t^{(m_k)}_j}
	\langle v_0,  (z^{(m_k)}_j-z^{(m_k)}_{j-1})\frac{t- t^{(m_k)}_j}{h_{m_k}}\rangle\,dt$$
	$$= \sum_{j= j_1+1}^{j_2}\frac{h_{m_k}}{2}
	\langle v_0,  (z^{(m_k)}_{j-1}-z^{(m_k)}_j)\rangle
	= \frac{h_{m_k}}{2}\langle v_0,  (z^{(m_k)}_{j_1}-z^{(m_k)}_{j_2})\rangle.$$
	By the boundedness of $\{z^{(m)}_j\}$, it converges to zero.
	
	By using a density argument, we can conclude that 
	$$\int^T_{-T}\langle v(t),  u^{(m_k)}(t)- \tilde{u}^{(m_k)}(t)\rangle\,dt\to 0$$
	for general $v\in L^2(-T, T; H^{-s}(\Omega)).$
\end{proof}

\bibliographystyle{plain}
\bibliography{Reference3}
\end{document}